\documentclass[11pt,a4paper,reqno]{amsproc}

\makeatletter
\usepackage{datetime}
\usepackage[pdfborder={0 0 0}]{hyperref}
\usepackage{upref,cases}
\hypersetup{citebordercolor=.1 .6 1}

\usepackage[T1]{fontenc}
\usepackage[utf8]{inputenc}

\usepackage[final]{graphicx}

\usepackage{hyperref}

\usepackage{geometry}
\IfFileExists{mymtpro2.sty}{\usepackage[subscriptcorrection]{mymtpro2}

  \def\cup{\cupprod}
  \def\cap{\capprod}
  \def\bigcup{\bigcupprod}
  
  \def\bigcupdisjoint{\mathop{\kern10pt\raisebox{4pt}{$\cdot$}\kern-12pt\bigcup}\limits}
}%
	{\usepackage{amssymb,bm,dsfont}

}

\numberwithin{equation}{section}
\newtheoremstyle{ttheorem}%
       {1.8ex\@plus1ex}                
       {2.1ex\@plus1ex\@minus.5ex}      
       {\itshape}           
       {0pt}                   
       {\bfseries}          
       {.}                  
       {.5em}               
       {}                

\newtheoremstyle{ddefinition}%
       {1.8ex\@plus1ex}                
       {2.1ex\@plus1ex\@minus.5ex}      
       {}           
       {0pt}                   
       {\bfseries}           
       {.}                  
       {.5em}               
       {}                

\newtheoremstyle{rremark}%
       {1.8ex\@plus1ex}                
       {2.1ex\@plus1ex\@minus.5ex}      
       {\normalfont}        
       {0pt}                   
       {\bfseries}           
       {.}                  
       {.5em}               
       {}                   

\theoremstyle{ttheorem}
\newtheorem{theorem}{Theorem}[section]
\newtheorem{lemma}[theorem]{Lemma}
\newtheorem{proposition}[theorem]{Proposition}
\newtheorem{corollary}[theorem]{Corollary}

\theoremstyle{ddefinition}

\theoremstyle{rremark}
\newtheorem{remark}[theorem]{Remark}
\newtheorem{myremarks}[theorem]{Remarks}
\newtheorem{myexamples}[theorem]{Examples}

\newenvironment{remarks}{\begin{myremarks}\begin{nummer}}%
    {\end{nummer}\end{myremarks}}
    {\end{nummer}\end{myexamples}}

\newcounter{numcount}
\newcommand{\labelnummer}{(\roman{numcount})}%

\makeatletter
\providecommand{\showkeyslabelformat}[1]{\relax}        
\let\mysaveformat\showkeyslabelformat                   %
\def\myformat#1{\raisebox{-1.5ex}{\mysaveformat{#1}}}   %

\newenvironment{nummer}%
  {\let\curlabelspeicher\@currentlabel%
    \begin{list}{\textup{\labelnummer}}%
      {\usecounter{numcount}\leftmargin0pt%
        \topsep0.5ex\partopsep2ex\parsep0pt\itemsep0ex\@plus1\p@%
        \labelwidth2.5em\itemindent3.5em\labelsep1em%
      }%
    \let\saveitem\item%
    \def\item{\saveitem%
      \def\@currentlabel{\curlabelspeicher\kern.1em\labelnummer}}%
    \let\savelabel\label%
    \def\label##1{{\ifnum\thenumcount=1\let\showkeyslabelformat\myformat\fi\savelabel{##1}}%
										{\def\@currentlabel{\labelnummer}%
									 	\let\showkeyslabelformat\@gobble
									 	\savelabel{##1item}%
										}%
	   							}%
  }{\end{list}}%

  {\let\curlabelspeicher\@currentlabel%
    \begin{list}{\textup{\labelnummer}}%
      {\usecounter{numcount}\leftmargin0pt%
        \topsep0.5ex\partopsep2ex\parsep0pt\itemsep0ex\@plus1\p@%
        \labelwidth2.5em\itemindent0em\labelsep1em%
        \leftmargin2.5em}%
    \let\saveitem\item%
    \def\item{\saveitem%
      \def\@currentlabel{\curlabelspeicher\kern.1em\labelnummer}}%
    \let\savelabel\label%
    \def\label##1{{\ifnum\thenumcount=1\let\showkeyslabelformat\myformat\fi\savelabel{##1}}%
										{\def\@currentlabel{\labelnummer}%
									 	\let\showkeyslabelformat\@gobble
									 	\savelabel{##1item}%
										}%
    							}%
  }{\end{list}}%

\def\section{\@startsection{section}{1}%
  \z@{1.3\linespacing\@plus\linespacing}{.5\linespacing}%
  {\normalfont\bfseries\centering}}

\def\subsection{\@startsection{subsection}{2}%
  \z@{.8\linespacing\@plus.5\linespacing}{-1em}%
  {\normalfont\bfseries}}

\def\nlsubsection{\@startsection{subsection}{2}%
  \z@{.8\linespacing\@plus.5\linespacing}{.1ex}%
  {\normalfont\bfseries}}

\let\@afterindenttrue\@afterindentfalse%

\renewenvironment{proof}[1][\proofname]{\par \normalfont
  \topsep6\p@\@plus6\p@ \trivlist 
  \item[\hskip\labelsep\scshape
    #1\@addpunct{.}]\ignorespaces
}{%
  \qed\endtrivlist
}

\def\ps@firstpage{\ps@plain
  \def\@oddfoot{\normalfont\scriptsize \hfil\thepage\hfil
     \global\topskip\normaltopskip}%
  \let\@evenfoot\@oddfoot
  \def\@oddhead{
    \begin{minipage}{\textwidth}
      \normalfont\scriptsize
      \emph{\insertfirsthead}
    \end{minipage}}
  \let\@evenhead\@oddhead 
}

\def\insertfirsthead{}

\def\@cite#1#2{{%
 \m@th\upshape\mdseries[{#1}{\if@tempswa, #2\fi}]}}
\renewcommand{\H}{\mathcal{H}}

\newcommand{\C}{\mathbb{C}}
\newcommand{\N}{\mathbb{N}}

\newcommand{\R}{\mathbb{R}}

\renewcommand{\le}{\leqslant}

\DeclareMathOperator{\tr}{tr}

\DeclareMathOperator{\dist}{dist}

\providecommand{\wtilde}[1]{\widetilde{#1}}

\providecommand{\bigcupdisjoint}{\mathop{\kern7pt\raisebox{6pt}{$\cdot$}\kern-9.5pt\bigcup}\limits}

\let\<\langle
\let\>\rangle

\providecommand{\Bigparens}[1]{\Bigl(#1\Bigr)}

\DeclareMathOperator{\Borel}{Borel}      

\newcommand{\1}{1}

\newcommand{\upd}{\mathrm{d}}
\renewcommand{\d}{\upd}   
\newcommand{\dx}{\d x}

\newcommand{\dy}{\d y}

\DeclareMathOperator{\ran}{ran}          
\newcommand{\hairspace}{\kern .04167em}

\renewcommand{\S}{\mathcal{S}}

\DeclareMathOperator{\TextIm}{Im}

\renewcommand{\Im}{\TextIm}

\def\clap#1{\hbox to 0pt{\hss#1\hss}}

\def\bra{\makeatletter\@ifstar\@bra\@@bra}
\def\@bra#1{\hairspace #1\>}
\def\@@bra#1{\lvert\@bra{#1}}
\def\ket{\makeatletter\@ifstar\@ket\@@ket}
\def\@ket#1{\<#1\hairspace}
\def\@@ket#1{\@ket{#1}\rvert}

\makeatother
\newcommand{\beq}{\begin{equation}}
\newcommand{\eeq}{\end{equation}}

\newcommand{\id}{\mathbf{1}}

\begin{document}

\title[Fredholm determinants related to spectral projections]{On an integral formula for Fredholm determinants related to pairs of spectral projections}

\author{Martin Gebert}
\address{King's College London\\Department of Mathematics\\ Strand London\\ WC2R 2LS, UK}

\curraddr{ \\ Queen Mary University of London\\ School of Mathematical Sciences\\ London, E1 4NS,
UK}
\thanks{M.G. was  supported by the DFG under grant GE 2871/1-1.}
\email{m.gebert@qmul.ac.uk}


\begin{abstract}
We consider Fredholm determinants of the form identity minus product of spectral projections corresponding to isolated parts of the spectrum of a pair of self-adjoint operators.  We show an identity relating such determinants to an integral over the spectral shift function in the case of a rank-one perturbation. More precisely, we prove
$$
-\ln \left(\det \big(\mathbf{1} -\mathbf{1} _{I}(A) \mathbf{1}_{\mathbb R\backslash I}(B)\mathbf{1}_{I}(A)\big) \right)
=
\int_I \text{d} x \int_{\mathbb R\backslash I} \text{d} y\, \frac{\xi(x)\xi(y)}{(y-x)^2},
$$
where $\mathbf{1}_J (\cdot)$ denotes the spectral projection of a self-adjoint operator on a set $J\in \text{Borel}(\mathbb R)$.
The operators $A$ and $B$ are self-adjoint, bounded from below and differ by a rank-one perturbation and $\xi$ denotes the corresponding spectral shift function. The set $I$ is a union of intervals on the real line such that its boundary lies in the resolvent set of $A$ and $B$ and such that the spectral shift function vanishes there i.e. $I$ contains isolated parts of the spectrum of $A$ and $B$. We apply this formula to the subspace perturbation problem. 
\end{abstract}

\maketitle

\section{Introduction}

In this paper we study Fredholm determinants of products of spectral projections corresponding to isolated parts of the spectrum of self-adjoint operators $A$ and $B$ which differ by a rank-one operator.  We are interested in Fredholm determinants of the form 
\beq\label{intro:eq3}
\det\big(\id - \id_I(A) \id_{{\R\backslash I}}(B) \id_I(A)\big)
\eeq
where $\id_J(\cdot)$ denotes the spectral projection of a self-adjoint operator corresponding to a set $J\in \Borel(\R)$. $I$ is a finite union of intervals such that its boundary lies in the intersection of the resolvent sets, i.e. $\partial I\subset \varrho(A)\cap\varrho(B)$. This together with a trace-class assumption on $B-A$ ensures that $\id_I(A) \id_{{\R\backslash I}}(B) \id_I(A)$ is trace-class and \eqref{intro:eq3} is actually well-defined. 
We investigate in this paper if there is an elementary integral representation for the determinant \eqref{intro:eq3}. The answer is yes in the case of a \textit{rank-one perturbation} and it is given in terms of a rather simple integral depending on the \textit{spectral shift function} $\xi$ of the pair $A$ and $B$. More precisely, we prove in Theorem~\ref{thm:main} the integral representation
\beq\label{intro:thm1:eq2a}
-\ln \left(\det \big(\id - \id_{I}(A) \id_{{\R\backslash I}}(B)\id_{I}(A)\big) \right)
=
\int_I \d x \int_{{\R\backslash I}} \d y \frac{\xi(x)\xi(y)}{(y-x)^2},
\eeq
where we assume that $B-A$ is \textit{rank one} and that the spectral shift function vanishes on $\partial I$. To the best of the author's knowledge this identity is new. Apart from products of spectral projections, \eqref{intro:thm1:eq2a}  directly implies a similar result for differences of spectral projections $\id_I(A) -\id_I(B)$, see Corollary \ref{corollary2}.

In the literature, 
such determinants are sometimes called section determinants and are studied for example in \cite{MR1694668, MR2121573}. There the Fredholm determinant \eqref{intro:thm1:eq2a} is computed abstractly in terms of solutions of particular operator-valued Wiener-Hopf equations. However, apart from existence results the solutions to this equations are not found explicitly. 

Over the last years further interest in Fredholm determinants of the form \eqref{intro:eq3} emerged in mathematical physics, see \cite{artAOC2014KOS,MR3212877,artAOC2015GKMO,MR3341967,MR3376020,DGM}. 
The determinant \eqref{intro:eq3} appears when computing the thermodynamic limit of the scalar product of two non-interacting fermionic many-body ground states filled up to the Fermi energy $E\in \R$. In this case, the underlying one-particle operators are given by a pair of Schr\"odinger operators whose difference is relatively trace class, i.e. both systems differ only locally. The identity  \eqref{intro:thm1:eq2a} gives a tool to compute this many-body scalar product explicitly if the one-particle operators differ by a rank-one perturbation and
 the Fermi energy lies in a spectral gap. This might be the case for periodic Schr\"odinger operators or Schr\"odinger operators with constant magnetic fields. 
 We will not go into more details about this problem here and refer to the aforementioned papers for further reading. 

Apart from the latter motivation, we apply the integral formula to the subspace perturbation problem. This constitutes of finding a bound on the operator norm of the difference of certain pairs of spectral projections. We refer to Section \ref{sec:subspace} for further explanations.

\section{Model and results}
Let $\H$ be a separable Hilbert space and $A$ be a self-adjoint operator on $\H$ which is bounded from below, i.e. there exists a constant $c\in \R$ such that $A> c$.  Let $B$ be a rank-one perturbation of $A$, i.e. we define
\begin{equation}\label{determinant:def1}
 B:=A+V,
\end{equation}
where $V:=|\phi\> \< \phi|$ for some $\phi\in\H$. Throughout, we use the notation $\sigma(\cdot)$ for the spectrum and $\varrho(\cdot)$ for the resolvent set of an operator and we write $\id$ for the identity operator on $\H$. Moreover, we denote by  $\partial J$ the boundary of a set $J\in \Borel(\R)$. 
We write $\S^1$ for the set of all trace-class operators on $\H$. 

The spectral shift function of the pair $A$ and $B$ is defined by the limit
\beq\label{SSF:def}
\xi(E) := \lim_{\varepsilon\to 0} \frac 1 \pi \text{arg}  \Big(1+\big\<\phi, \frac 1 {A-E-i\varepsilon} \phi\big\>\Big),
\eeq
where $\text{arg}(z)\in [0,2\pi)$ denotes the argument of a complex number $z\in \C$. The latter limit is well-defined and exists for Lebesgue almost all $E\in\R$, see  \cite[Sec. 11]{MR2154153}. The non-negativity of $\lim_{\varepsilon\to 0}\Im \big\<\phi, \frac 1 {A-E-i\varepsilon} \phi\big\> \geq 0$ implies that $0\leq\xi\leq 1$. Additionally, standard results provide that $\xi\in L^1(\R)$ with norm $\|\xi\|_{L^1(\R)}= \int_\R \d x\, \xi(x) = \tr (V)= \<\phi,\phi\>$. A priori the spectral shift function is defined only for Lebesgue almost all $E\in\R$. However, the spectral shift function is defined for all energies $E\in \varrho(A)\cap \rho(B)\cap \R$ and is given for such $E$ by
\beq\label{SSF:resolvent}
\xi(E):= \tr \big(\id_{(-\infty,E)}(A) - \id_{(-\infty,E)}(B)\big),
\eeq
where the later difference is trace-class, and $\xi(E)\in \{0,1\}$, see e.g. \cite[Sec. 11]{MR2154153}. In particular, $\xi$ is constant with values in $\{0,1\}$ on any connected component of $\rho(A)\cap\rho(B)\cap \R$. 
The definition \eqref{SSF:def} of the spectral shift function is specific to rank-one perturbations. Further properties in the rank-one case can be found in \cite[Sec. 11]{MR2154153}, whereas we refer to \cite{MR1180965} for the general theory beyond rank-one perturbations. 

The second central quantity in this paper is the Fredholm determinant. We briefly recall the definition. For an operator $K\in \S^1$ with eigenvalues $(c_n)_{n\in\N}$, listed according to their algebraic multiplicities, we define the Fredholm determinant of $\id -K$ by the product
\beq
\det \big( \id - K\big) := \prod_{n\in\N} (1 - c_n).
\eeq
The latter product is well defined by the trace-class assumption on $K$, see \cite[Sect.\ XIII.17]{MR0493421}. Fredholm determinants share most properties with the usual determinant in finite dimensional spaces. Especially, we will need the multiplicativity of Fredholm determinants. 

Now, we are ready to state the main result of the paper, which relates Fredholm determinants of products of spectral projections corresponding to isolated spectral subsets to a particular integral over the spectral shift function:

\begin{theorem}\label{thm:main}
Let $N\in\N$ and 
\beq\label{defI}
I:=\bigcup_{i=1}^N[E_{2i-1},E_{2i}]
\eeq
where 
$-\infty\leq E_1\leq E_2<... <E_{2N-1}\leq E_{2N}\leq\infty$.
We further assume that the boundary of $I$ satisfies
\beq\label{assum:specgap}
\partial I\subset \rho(A)\cap \rho(B)
\eeq
and that the spectral shift function vanishes on $\partial I$, i.e. for all $x\in \partial I$
\beq\label{thm:assump}
\xi(x) = 0.
\eeq
Then, $ \id_{I}(A) \id_{I^c}(B)\id_{I}(A)\in \S^1$ and the identity 
\beq\label{thm1:eq2a}
-\ln \left(\det \big(\id - \id_{I}(A) \id_{I^c}(B)\id_{I}(A)\big) \right)
=
\int_I \d x \int_{I^c} \d y \frac{\xi(x)\xi(y)}{(y-x)^2}
\eeq
holds, where $I^c:=\R\backslash I$. 
\end{theorem}

\begin{remarks}
\item
The spectral shift function $\xi$ is initially only defined for Lebesgue almost all $E\in \R$, but $\xi$ makes sense for all $E\in\partial I\subset \rho(A)\cap \rho(B)$, see identity \eqref{SSF:resolvent}.
\item
Both quantities in \eqref{thm1:eq2a} are well-defined under the assumptions of the theorem: Loosely speaking, the trace-class property follows from $\partial I$ being in a spectral gap of both operators $A$ and $B$ and $B-A\in\S^1$. The convergence of the integral follows from \eqref{thm:assump}. 
\item
One word about the notation. If we consider
 $\id_{I}(A) \id_{I}(B)\id_{I}(A): \mathcal H' \to \mathcal H'$ for $\mathcal H':= \ran(\id_I(A))$, the latter determinant equals 
\beq
\det \big(\id - \id_{I}(A) \id_{I^c}(B)\id_{I}(A)\big) = 
\det \big(\id_{I}(A) \id_{I}(B)\id_{I}(A)\big|_{\mathcal H'}\big).
\eeq
Thus, one can formulate the result as well in the notation of restricted products of spectral projections. 
Such determinants are sometimes called section determinants \cite{MR1694668, MR2121573}. 
\end{remarks}

\begin{corollary}\label{corollary}
Under the assumptions of Theorem \ref{thm:main} we obtain $\id_{I^c}(A) \id_{I}(B) \id_{I^c}(A)\in \S^1$ and the identity
\beq\label{thm1:eq2b}
-\ln \left(\det \big(\id - \id_{I^c}(A) \id_{I}(B) \id_{I^c}(A)\big) \right)
=
\int_I \d x \int_{I^c} \d y \frac{\xi(x)\xi(y)}{(y-x)^2}
\eeq
holds.
\end{corollary}

A simple calculation shows that differences of spectral projections can be expressed as products of spectral projections and the following holds:

\begin{corollary}\label{corollary2}
Under the assumptions of Theorem \ref{thm:main} we obtain $\big(\id_I(A) -\id_I(B)\big)^2\in\S^1$ and the identity
\beq\label{thm1:eq3b}
-\ln \big( \det\big(\id - \big(\id_I(A) -\id_I(B)\big)^2\big) \big)
=
2 \int_I \d x \int_{I^c} \d y\, \frac{\xi(x)\xi(y)}{(y-x)^2}
\eeq
holds.  
\end{corollary}

The proof is short and follows directly from the previous results. Therefore, we do not postpone it for later and present it straight away: 

\begin{proof}[Proof of Corollary \ref{corollary2}]
A straight forward calculation shows that the identity
\beq
\id_I(A) -\id_I(B) =  \id_I(A)\id_{I^c}(B) - \id_{I^c}(A)\id_I(B)
\eeq
holds and therefore multiplying the latter with its adjoint we obtain
\begin{align}\label{proj-identity}
\big(\id_I(A) -\id_I(B) \big)^2
&= \big(\id_I(A) -\id_I(B) \big) \big(\id_I(A) -\id_I(B) \big)^*  \notag\\
&= \id_I(A) \id_{I^c}(B) \id_I(A) + \id_{I^c}(A) \id_I(B) \id_{I^c}(A).
\end{align}
Theorem \ref{thm:main}, Corollary \ref{corollary} and the multiplicativity of Fredholm determinants give the result.
\end{proof}

\begin{remark}
In particular, formula \eqref{thm1:eq3b} holds for Fermi projections, i.e. choosing $I=(-\infty,E]$. In this situation, the above determinant \eqref{thm1:eq3b} is related to the scalar product of the ground states of two non-interacting fermionic systems at Fermi energy $E$, see \cite{artAOC2015GKMO} and references cited therein. The above formula can be used to compute this scalar product exactly in cases where the Fermi energy lies in a spectral gap. 
\end{remark}

Considering the results of Theorem \ref{thm:main}, it is natural to ask if condition \eqref{thm:assump} is necessary for formula \eqref{thm1:eq2a} or \eqref{thm1:eq3b} to hold. We will discuss this in the following:

\begin{proposition}\label{prop1}
Assume that $I$ is given by \eqref{defI} and satisfies \eqref{assum:specgap}. Then, $\id_I(A) -\id_I(B)\in \S^1$. If $\tr\big(\id_I(A) -\id_I(B)\big) > 0$
the identity
\beq\label{prop1:eq1}
\det \big(\id - \id_{I}(A) \id_{I^c}(B)\id_{I}(A)\big)= 0 
\eeq
holds, whereas if $tr\big(\id_I(A) -\id_I(B)\big)<0$ then
\beq\label{prop1:eq2} 
\det \big(\id - \id_{I^c}(A) \id_{I}(B) \id_{I^c}(A)\big)  = 0.
\eeq
In particular, $\tr\big(\id_I(A) -\id_I(B)\big) \neq 0$ implies $ \det\big(\id - \big(\id_I(A) -\id_I(B)\big)^2\big)=0$. 
\end{proposition}
The above is an immediate consequence of the theory of indices of pairs of spectral projections \cite{MR1262254}. The condition $\tr\big(\id_I(A) -\id_I(B)\big)\neq 0$ is opposite to \eqref{thm:assump} in the sense that under assumption \eqref{thm:assump} we obtain $\tr\big(\id_I(A) -\id_I(B)\big)=0$. The condition $\tr\big(\id_I(A) -\id_I(B)\big)=0$ alone should imply an integral formula for our Fredholm determinant. However, such formulas must in general be different from \eqref{thm1:eq2a} as we see in the next theorem: 

\begin{theorem}\label{prop2}
Let $N\in\N$ and 
$I:=\bigcup_{i=1}^N[E_{2i-1},E_{2i}]$
where 
$-\infty\leq E_1\leq E_2<... <E_{2N-1}\leq E_{2N}\leq\infty$.
We assume that the boundary of $I$ satisfies
$\partial I\subset \rho(A)\cap \rho(B)$
and that the spectral shift function satisfies for all $x\in \partial I$
\beq\label{thm:assump2}
\xi(x) = 1.
\eeq
Then, $ \id_{I}(A) \id_{I^c}(B)\id_{I}(A)\in\S^1$ and the identity
\beq\label{thm1:eq222}
-\ln \left(\det \big(\id - \id_{I}(A) \id_{I^c}(B)\id_{I}(A)\big) \right)
=
\int_I \d x \int_{I^c} \d y\, \frac{\big(\xi(x)- 1\big)\big( \xi(y) - 1\big)}{(y-x)^2}
\eeq
holds.
\end{theorem}

We end this section with several general remarks concerning extensions of the above integral formulas to higher-rank perturbations.  

\begin{remarks}\label{remark}
\item
Since the spectral shift function makes sense for general trace-class perturbations, one can define the above integrals as well for such perturbations. However, we emphasise that Theorem \ref{thm:main} and its descendants are only valid for rank-one perturbations in the particular form stated. This can already be seen in the matrix case. Let
\beq
A:=
\begin{pmatrix}
0 & 0 \\ 0 & 3
\end{pmatrix}
\quad\text{and}\quad
B:=
\begin{pmatrix}
0 & 0 \\ 0 & 3
\end{pmatrix}
+
\begin{pmatrix}
1 & 0 \\ 0 & 1
\end{pmatrix}.
\eeq
Then, $A$ and $B$ differ by a rank-two perturbation. For $I:=[-1,2]$ the difference satisfies $\id_I(A) -\id_I(B)= 0$ and accordingly
$\det\big(\id - \big(\id_I(A) -\id_I(B)\big)^2\big) = 1$. 
Hence, 
\beq
-\ln \big(\det\big(\id - \big(\id_I(A) -\id_I(B)\big)^2\big) \big) =0.
\eeq 
On the other hand, the corresponding integral does not vanish, i.e.
\beq
 \int_I \d x \int_{I^c} \d y\, \frac{\xi(x) \xi(y)}{(y-x)^2} =  \int_0^1 \d x \int_{3}^4 \d y\, \frac{1}{(y-x)^2}> 0.
\eeq
However, it is an interesting question if a similar identity to  \eqref{thm1:eq3b} holds beyond rank one perturbations.
\item 
On a more abstract level one reason for the Fredholm determinant \eqref{thm1:eq2a} to be actually computable is that $\id_{I}(A)\id_{I^c}(B)\id_{I}(A)$ is an integrable operator in the sense of \cite{MR1469319}. Let us illustrate this very briefly. For this, we assume that the spectral measures
\beq
\mu_A(\,\cdot\,) :=\<\phi, \id_{(\cdot)}(A) \phi\>\quad\text{and}\quad \mu_B(\,\cdot\,):=\<\phi, \id_{(\cdot)}(B) \phi\>
\eeq
are absolutely continuous with densities $\frac{\d }{\d x}\mu_A= f$ and $\frac{\d }{\d x}\mu_B= g$.
Moreover, we assume that $\phi$ is cyclic with respect to $A$. Hence, the mappings
\beq
\begin{aligned}
&U^*:\H\to L^2(\sigma(A)),\quad h(A)\phi\mapsto h f^{1/2}\\
&V^*:\H\to L^2(\sigma(B)),\quad h(B)\phi\mapsto h g^{1/2}
\end{aligned}
\eeq
are well-defined unitaries, where $L^2(I)$ denotes the space of square integrable functions on $I\in\Borel(\R)$ w.r.t. Lebesgue measure.
Then we obtain 
\beq
U^*\id_{I}(A)\id_{I^c}(B)V = K,
\eeq
where $K:L^2(I^c\cap \sigma(B))\to L^2(I\cap \sigma(A))$ is an integral operator with kernel 
\beq
K(x,y) = \frac{f^{1/2}(x)g^{1/2}(y)}{x-y}\qquad \text{for}\ x\in I,\ y\in I^c,
\eeq
see \cite[Thm. 2.1]{MR2540995}.
Using this we obtain that $\id_{I}(A)\id_{I^c}(B)\id_{I}(A)$ is unitarily equivalent to the integral operator  $R:=K K^*: L^2(I\cap \sigma(A))\to L^2(I\cap \sigma(A))$ with kernel 
\beq\label{riemann-hilbert}
R(x,y):= \frac{f^{1/2}(x)G(y) - f^{1/2}(y)G(x)}{x-y} \quad\text{for}\quad x,y\in I\cap \sigma(A)
\eeq
where $G(x):= \int_{I^c} \d z\, \frac{g(z)}{x-z} f^{1/2}(x)$. 
This operator is integrable because $f^{1/2}(x)G(x)-f^{1/2}(x)G(x)=0$. Such operators have the property that one can invert $\id - R$ on an abstract level solving a matrix-valued Riemann-Hilbert problem, see \cite{MR1469319}. 
For rank-k perturbations, the product $\id_{I}(A)\id_{I^c}(B)\id_{I}(A)$ can, under some assumptions, still be identified by  a more complicated integrable operator. Hence, we conjecture in the case of a trace-class perturbation a formula of the form 
\beq\label{intfomula:general}
-\ln \left(\det \big(\id - \id_{I}(A) \id_{I^c}(B)\id_{I}(A)\big) \right)
=
2 \int_I \d x \int_{I^c} \d y\, \frac{H(x,y)}{(y-x)^2}
\eeq
for some function $H:\R^2\to \C$.

Even though our operator admits this integrable structure, we are not using it directly in the proof and we are not solving any matrix-valued Riemann Hilbert problem. We rather prove the theorem in a more elementary way, first for operators with discrete spectrum using the Cauchy determinant formula and then taking the limit in a suitable way. 
\end{remarks}

\subsection{An Application: The subspace perturbation problem}\label{sec:subspace}

Let $A$ be a self-adjoint operator such that its spectrum consists of two sets $\Sigma$ and $\sigma(A)\backslash\Sigma$ which are separated by a distance $\delta>0$. Now, we perturb $A$ by a self-adjoint perturbation $V$ of norm $\|V\|<\delta/2$. In this way $V$ does not close the gaps in the spectrum. We introduce the enlarged set
\beq\label{subspace:thm:def}
\wtilde\Sigma_{\delta/2} := \{x\in \R:\, \dist(x,\Sigma)<\delta/2\}.
\eeq
Then \cite{MR1991758} posed the question if this is sufficient to imply $\| \id_\Sigma(A) - \id_{\wtilde\Sigma_{\delta/2}} (B) \| < 1$ and proved that this is indeed the case for $\Sigma$ being a convex set in $\sigma(A)$. 
Even though there was some progress for general sets $\Sigma$,  it  seems to be an open problem to prove this for general sets. For the moment, one has to assume that $\|V\|< c\,\delta$ for some explicit $c<0.46$, see \cite{MR3079864,MR3231246,MR3420326} and references cited therein. Theorem~\ref{thm:main} provides a tool to prove this for arbitrary sets $\Sigma$ in the very special case of $V$ being a rank-one perturbation. 
In the case of a non-negative perturbation the spectrum is moved in a definite direction by the perturbation. Hence, in our case it is more convenient to work with the set
\beq\label{subspace:thm:def2}
\Sigma_{\delta} := \Sigma + [0,\delta]:= \{x+\wtilde \delta:\, x\in \Sigma,\, \wtilde \delta\in [0,\delta] \}.
\eeq

\begin{corollary}\label{subspace:thm}
Let $A, B$ be self-adjoint and bounded with $B-A=|\phi\>\<\phi|:=V$ for some $\phi\in\H$ and 
$\Sigma\subset \sigma(A)$
such that 
\beq\label{subspace:thm:eq}
\dist(\Sigma, \sigma(A)\backslash \Sigma)= \delta>0
\eeq
and $\|V\|=\|\phi\|^2<\delta$. Let $\Sigma_{\delta}$ be defined by \eqref{subspace:thm:def2} then 
\beq\label{subspace:bound}
\| \id_\Sigma(A) - \id_{\Sigma_{\delta}} (B) \| < 1. 
\eeq
\end{corollary}

It is of certain interest to give a quantitative bound in terms of the distance $\delta$ and the operator norm of $V$. The above result is rather indirect and doesn't provide this. With the method used in the above proof one cannot expect to obtain an optimal quantitative bound because we take all eigenvalues of $\id_\Sigma(A) - \id_{\Sigma_{\delta}} (B)$ into account instead of only the relevant one with the biggest modulus. But using different methods, one can prove an optimal bound  in the case of a rank-one perturbation:

\begin{theorem}\label{subspace:thm2}
Let $A, B$ be self-adjoint and bounded with $B-A=|\phi\>\<\phi|:=V$ for some $\phi\in\H$ and 
$\Sigma\subset \sigma(A)$
such that 
\beq\label{subspace:thm:eq2}
\dist(\Sigma, \sigma(A)\backslash \Sigma)= \delta>0
\eeq
and $\|V\|=\|\phi\|^2<\delta$. Let $\Sigma_{\delta}$ be defined by \eqref{subspace:thm:def2} then 
\beq
\| \id_\Sigma(A) - \id_{\Sigma_{\delta}} (B) \| < \| V\|/ \delta. 
\eeq
\end{theorem}

Even though the above theorem makes Corollary \ref{subspace:thm} obsolete, the argument in Corollary \ref{subspace:thm} may help to prove the bound \eqref{subspace:bound} in more general cases provided an integral formula exists for more general perturbations, see \eqref{intfomula:general}.  
We emphasise that our proof of the sharp result Theorem \ref{subspace:thm2} fails for perturbations other than rank one. 

\section{Proofs}

\subsection{Proof of Theorem \ref{thm:main} for finite-rank operators}\label{subsec:matrices}

Throughout the proof we assume that $\phi$ is cyclic with respect to $A$. If this is not the case from the beginning we have to restrict ourselves to the cyclic subspace which we omit to keep the notation simple. We first prove the result for matrices. For $M\in\N$, we denote by $\C^{M\times M }$ the set of all $M\times M$ matrices with $\C$-valued entries. 

Let $M\in\N$. From now on we assume in this section that $A,B\in \C^{M\times M}$, are self-adjoint and $B:=A+ |\phi\>\<\phi|$, where $\phi\in \C^M$.

\begin{lemma}\label{main:lem}
Under the assumptions of Theorem \ref{thm:main} on the set $I$, the identity
\begin{align}\label{thm1:eq2aa}
-\ln \left(\det \big(\id - \id_{I}(A) \id_{I^c}(B)\id_{I}(A)\big) \right)
=  \int_I \d x \int_{I^c} \d y\, \frac{\xi(x)\xi(y)}{(y-x)^2}
\end{align}
holds.
\end{lemma}
Similarly, we formulate Theorem \ref{prop2} for matrices:

\begin{lemma}\label{main:lem2}
Under the assumptions of Theorem \ref{prop2} on the set $I$, the identity
\beq\label{thm1:eq2}
-\ln \left(\det \big(\id - \id_{I}(A) \id_{I^c}(B)\id_{I}(A)\big) \right)=
 \int_I \d x \int_{I^c} \d y\, \frac{\big(\xi(x)- 1\big)\big( \xi(y) - 1\big)}{(y-x)^2}
\eeq
holds.
\end{lemma}

To prove this we need some auxiliary results. 
In the following, we write $( a_i)_{i=1}^M$ and $( b_i)_{i=1}^M$ for the sequences of eigenvalues of $A$ and $B$ ordered non-decreasingly. Since we restricted ourselves to cyclic $\phi$, the eigenvalues interlace  strictly, i.e.
\beq\label{ineq:eigenvalues}
 a_1<  b_1<   a_2< \cdots<  a_M < b_M.
\eeq 
We denote by $\left(\varphi_j\right)_{j=1}^M$ and $\left(\psi_k\right)_{k=1}^M$ the corresponding normalised eigenvectors of $A$, respectively, of $B$.
In particular, the spectral shift function is given by 
\beq
\xi(E)=\tr \big(\id_{(-\infty,E)}(A) - \id_{(-\infty,E)}(B)\big) =  \sum_{n=1}^M 1_{( a_n, b_n]}(E)
\eeq
for $E\in \rho(A)\cap\rho(B)\cap \R$, see \cite[Prop. 11.11]{MR2154153}. Here, with $1_I$ we denote the indicator function of the set $I\subset \R$ on $\R$. 

\begin{lemma}\label{Fredholm:Repr}
Let $I\in \Borel(\R)$ such that the index sets $J_I(A):= \big\{j\in \{1,...,M\}:  a_j\in I\big\}$ and $J_I(B):= \big\{k \in \{1,...,M\}:  b_k\in I\big\}$ have the same cardinality. 
Then
\begin{align}
\det \big(\id - \id_{I}(A) \id_{I^c}(B)\id_{I}(A)\big)
&= \det \big(\id - \id_{I^c}(A) \id_{I}(B) \id_{I^c}(A)\big)\notag\\
&= \big|\det\big(\big\<\varphi_j,\psi_k\big\>\big)_{ j\in J_I(A), k\in J_I(B) }\big|^2.
\end{align}
\end{lemma}

\begin{proof}[Proof of Lemma \ref{Fredholm:Repr}]
Define  $S:=\det \mathcal O$ where $\mathcal O$ is the $|J_I(A)|\times |J_I(B)|$ matrix
\beq
\mathcal O:= \big(\<\varphi_j, \psi_k\>\big)_{ j\in J_I(A), k\in J_I(B) }.
\eeq
Moreover, we use the abbreviations $P:=\id_{I}(A)$ and $Q:=\id_{I}(B)$.
  Then, the matrix entries of $\mathcal O\mathcal O^*$ and $\mathcal O^*\mathcal O$ read
  \begin{align}
  	(\mathcal O\mathcal O^*)_{jl} &= \sum_{k\in J_I(B)} \< \varphi_j, \psi_k\>\< \psi_k,\varphi_l\> 
  		= \< \varphi_j,P Q P\varphi_l\>, \\
  	(\mathcal O^*\mathcal O)_{jl} &= \sum_{k\in J_I(A)} \< \psi_j, \varphi_k\>\< \varphi_k,\psi_l\> 
			= \< \psi_j, Q P Q\psi_l\>.
  \end{align}
  Using the multiplicativity of the determinant, $|S|^2$ can be rewritten in two ways:
  \begin{align}
    \label{Fredholm:eq4}
    |S|^2 = \det ( \mathcal O\mathcal O^*)
	      &= \det \big( P Q P \big|_{\ran P}\big) \notag\\
	      &= \det \big( \id - P (\id- Q) P\big) ,
  \end{align}
  and likewise as
  \begin{align}
    \label{Fredholm:eq5}
    |S|^2 = \det (\mathcal O^*\mathcal O)
	      &=  \det \big( Q P Q  \big|_{\ran Q}\big) \notag\\
	      &=  \det\big(\id - Q(\id - P) Q \big) \notag\\
	      &=  \det\big(\id - (\id - P) Q (\id - P) \big).
  \end{align}
  The last equality follows from the fact that the non-zero singular values of $Q(\id - P)$ 
  coincide with the non-zero singular values of its adjoint $(\id - P)Q$. Now, equations 
  \eqref{Fredholm:eq4} and \eqref{Fredholm:eq5} give the assertion.
\end{proof}

\begin{lemma}\label{prod:thm1}
Let $J_1,J_2\subset \{1,...,M\}$ be two index sets with the same cardinality $N\in\N$, with $N\leq M$, i.e. $J_1=\{l_1,..,l_N\}$ and $J_2=\{m_1,...,m_N\}$ for some numbers $l_1< l_2 ...< l_N$ and $m_1 < m_2...< m_N$. Then,
 \begin{equation}\label{prod:rep}
  \Big|  \det\Bigparens{\<\varphi_{l_j}, \psi_{m_k}\>}_{1\le j,k\le N}\Big|^2=
  \prod_{j=1}^N\prod_{k= N+1}^M\frac{\left| b_{m_k}- a_{l_j}\right|\left| a_{l_k}- b_{m_j}\right|}{\left| a_{l_k}- a_{l_j}\right|\left| b_{m_k}- b_{m_j}\right|},
 \end{equation}
where $(l_k)_{k=N+1}^M$ is a sequence such that $\{l_1,...,l_N\} \cap \{l_{N+1},..., l_M\}= \emptyset$ and $\{l_1,..l_N\}\cup \{l_{N+1},...,l_M\}= \{1,...,M\}$. The sequence $(m_k)_{k=N+1}^M$ is defined accordingly.
\end{lemma}

\begin{remark}
If $J_1=J_2= \{1,2...,N\}$ formula \eqref{prod:rep} reads
\beq\label{prod:rep2}
  \Big|  \det\Bigparens{\<\varphi_j, \psi_k\>}_{j\in J_1, k\in J_2}\Big|^2=
  \prod_{j=1}^N\prod_{k=N+1}^M\frac{\left| b_k- a_j\right|\left| a_k- b_j\right|}{\left| a_k- a_j\right|\left| b_k- b_j\right|}.
\eeq
Such a formula is known in the physics literature in the context of scalar products of ground states of non-interacting fermionic systems and goes back at least to \cite{RevModPhys.62.929}. In this context it was also revisited in \cite{MR3405952}. 
\end{remark}

\begin{proof}[Proof of Lemma \ref{prod:thm1}] 
The strict interlacing of the eigenvalues, see \eqref{ineq:eigenvalues}, implies that none of the denominators in the above product vanishes and the product is well-defined. 
In the following, we assume w.l.o.g. that $J_1=J_2=\{1,...,N\}$. The general case follows from this case after relabelling. 
 The eigenvalue equations imply for all $j,k\in\N$
\begin{equation}
 \<\varphi_j,\psi_k\>=\frac{\<\varphi_j,\phi\>\<\phi,\psi_k\>}{ b_k- a_j}.
\end{equation}
This and the properties of the determinant give
\begin{align}
 \Big|  \det\Bigparens{\<\varphi_j, \psi_k\>}_{1\le j,k \le N}\Big|^2
=& \Big|  \det\Bigparens{\frac{\<\varphi_j,\phi\>\<\phi,\psi_k\>}{ b_k- a_j}}_{1\le j,k \le N}\Big|^2\nonumber\\
=& \prod_{j=1}^N\prod_{k=1}^N\left|\<\varphi_j,\phi\>\<\phi,\psi_k\>\right|^2 \Big|\det\Bigparens{\frac 1 { b_k- a_j}}_{1\le j,k \le N}\Big|^2.\label{determinant}
\end{align}
The remaining determinant $\det\big( \frac 1 { b_k- a_j}\big)_{1\le j,k \le N} $ is the determinant of a Cauchy matrix. Such Cauchy determinants can be computed explicitely, see e.g. \cite[Lem. 7.6.A]{zbMATH06257273}, and we end up with 
\begin{align}
\eqref{determinant}
= \prod_{j=1}^N\prod_{k=1}^N\left|\<\varphi_j,\phi\>\<\phi,\psi_k\>\right|^2 \frac{ \prod_{j,k=1,j\neq k}^N\left| b_k- b_j\right|\left| a_j- a_k\right|}{\prod_{j,k=1}^N \left| b_k- a_j\right|^2}.\label{det:eq3}
\end{align}
The remaining scalar products of the eigenvectors can be computed explicitly, see Lemma \ref{cor:det} in the appendix. This implies
\begin{align}\label{determinant3}
\eqref{det:eq3}
=&   
\prod_{k=1}^N 
 \prod_{\substack{l=1\\l\neq k}}^M\frac{ \left| a_l- b_k\right|}{ \left| b_l- b_k\right|}
 \prod_{j=1}^N 
 \prod_{\substack{l=1\\l\neq j}}^M\frac{ \left| b_l- a_j\right|}{\left| a_l- a_j\right|}\nonumber
 \prod_{\substack{j,k=1\\j\neq k}}^N\frac{\left| b_k- b_j\right|\left| a_j- a_k\right|}{\left| b_k- a_j\right|^2}\\
 =&
 \prod_{j=1}^N\prod_{k=N+1}^M \frac{\left| b_k- a_j\right|\left|  a_k- b_j\right|}{ | b_k- b_j|| a_j- a_k|}.
\end{align}
Now, the assertion follows from \eqref{determinant}, \eqref{det:eq3} and \eqref{determinant3}.
\end{proof}

\begin{proof}[Proof of Lemma \ref{main:lem}]
We define the sets $J_1:=\{m_1,m_2,...,m_N\}= \{k\in\{1,...,M\}: b_k\in I\}$ and $J_2:=\{l_1, l_2,...,l_N\} =\{j\in\{1,...,M\}: a_j\in I\}$ and the sequences $\{m_{N+1},...m_M\}$ and $\{l_{N+1},...,l_M\}$ as in Lemma~\ref{prod:thm1}. Assumption \eqref{thm:assump} implies that these index sets coincide, i.e. $J_1=J_2=:J$.
Lemma \ref{Fredholm:Repr} and Lemma 
~\ref{prod:thm1} yield
\begin{align}\label{proof:main:eq1}
\det \big(\id - \id_{I}(A) \id_{I^c}(B)\id_{I}(A)\big) 
&=
  \prod_{j=1}^N\prod_{k= N+1}^M\frac{\left| b_{m_k}- a_{l_j}\right|\left| a_{j_k}- b_{m_j}\right|}{\left| a_{l_k}- a_{l_j}\right|\left| b_{m_k}- b_{m_j}\right|}\notag\\
&=
  \prod_{j\in J}\prod_{k\notin J}\frac{\left| b_{k}- a_{j}\right|\left| a_{k}- b_{j}\right|}{\left| a_{k}- a_{j}\right|\left| b_{k}- b_{j}\right|},
\end{align}
where  and we used in the last line that $J_1=J_2=J$.

On the other hand, the spectral shift function is given by 
\beq
\xi(x) = \sum_{n=1}^M 1_{( a_n, b_n]}(x).
\eeq
Inserting this in the integral on the r.h.s of \eqref{thm1:eq2a}, we compute
\begin{align}
\int_I \d x \int_{I^c} \d y \frac{\xi(x)\xi(y)}{(y-x)^2}
 =&
\sum_{j=1}^M \sum_{k=1}^M \int_I \d x \int_{I^c} \d y\frac{1_{( a_j, b_j]}(x) 1_{( a_k, b_k]}(y)}{(y-x)^2}\notag\\
=&
\sum_{j\in J} \sum_{k \notin J} \int_{ a_j}^{ b_j}  \int_{ a_k}^{ b_k} \frac{\d x \d y}{(y-x)^2}.
\end{align}
Integrating the latter, gives
\beq
\sum_{j\in J} \sum_{k \notin J} \int_{ a_j}^{ b_j}  \int_{ a_k}^{ b_k} \frac{\d x \d y}{(y-x)^2}
= 
\sum_{j\in J}\sum_{k\notin J}\ln\left(\frac{\left| a_k- a_j\right|\left| b_k- b_j\right|}{\left| b_k- a_j\right|\left| a_k- b_j\right|}\right).
\eeq
This gives the assertion taking the logarithm in \eqref{proof:main:eq1}. 
\end{proof}

\begin{proof}[Proof of Lemma \ref{main:lem2}]
We define the sets $J_1:=\{m_1,m_2,...,m_N\}= \{k\in\{1,...,M\}: b_k\in I\}$ and $J_2:=\{l_1, l_2,...,l_N\} =\{j\in\{1,...,M\}: a_j\in I\}$ and the sequences $\{m_{N+1},...m_M\}$ and $\{l_{N+1},...,l_M\}$ as in Lemma~\ref{prod:thm1}. These sets have the same cardinality $|J_1|=|J_2|$ by \eqref{thm:assump} but $J_1\neq J_2$. However, both sets are still very close: Since the spectral shift function is $1$ on $\partial I$, we obtain 
\beq\label{proof:main:eq10}
J_1+1=\{m_1+1,m_2+1,...,m_N+1\}=J_2,
\eeq
 where we used the convention $M+1:=1$ here. We apply Lemma \ref{Fredholm:Repr} and Lemma \ref{prod:thm1} to the l.h.s. of \eqref{thm1:eq2} and
obtain the product representation 
\begin{align}\label{proof:main:eq1b}
\det \big(\id - \id_{I}(A) \id_{I^c}(B)\id_{I}(A)\big) 
&=
  \prod_{j=1}^N\prod_{k= N+1}^M\frac{\left| b_{m_k}- a_{l_j}\right|\left| a_{l_k}- b_{m_j}\right|}{\left| a_{l_k}- a_{l_j}\right|\left| b_{m_k}- b_{m_j}\right|}\notag\\
  &=
  \prod_{j\in J_1}\prod_{k\notin J_1}\frac{\left| b_{k}- a_{{j}+1}\right|\left| a_{k+1}- b_{j}\right|}{\left| a_{k+1}- a_{j+1}\right|\left| b_{k}- b_{j}\right|},
\end{align}
where we used \eqref{proof:main:eq10} in the last line. 

On the other hand, we compute
\begin{align}
\xi(x)-1 
= - \sum_{n=1}^{M-1} 1_{( b_n, a_{n+1}]}(x) - 1_{\R\backslash ( a_1, b_M]}(x)
= -\sum_{n=1}^M 1_{( b_n, a_{n+1}]}(x),
\end{align}
where $1_{( b_M, a_{M+1}]}:= 1_{\R\backslash ( a_1, b_M]}(x)$.
Inserting this in the r.h.s of \eqref{thm1:eq2} and 
evaluating the integral, gives
\begin{align}
\int_I \d x \int_{I^c} \d y \frac{(\xi(x)-\1)(\xi(y)-1)}{(y-x)^2}
 =&
\sum_{j=1}^M \sum_{k=1}^M \int_I \d x \int_{I^c} \d y\frac{1_{( b_j, a_{j+1}]}(x) 1_{( b_k, a_{k+1}]}(y)}{(y-x)^2}\notag\\
=&
\sum_{j\in J_1}\sum_{k \notin J_1} \int_{ b_j}^{ a_{j+1}}  \int_{ b_k}^{ a_{k+1}} \frac{\d x \d y}{(y-x)^2}
\notag\\
=&
\sum_{j\in J_1}\sum_{k\notin  J_1}\ln\left(\frac{\left| a_{k+1}- a_{j+1}\right|\left| b_k- b_j\right|}{\left| b_k- a_{j+1}\right|\left| a_{k+1}- b_j\right|}\right).
\end{align}
This and taking the logarithm in \eqref{proof:main:eq1b} implies the assertion. 
\end{proof}

\subsection{General bounded operators: An approximation argument}\label{approx}

In this section we lift the results obtained for matrices in the previous section to general bounded operators. To do so, we define particular finite-rank approximations which respect the spectral gaps of the limiting operators up to a security distance.

We assume in this section that $A$ and $B$ are bounded self-adjoint operators which satisfy \eqref{determinant:def1} and $I$ be such that the assumptions of Theorem \ref{thm:main} are met. 
Additionally, we assume w.l.o.g. that $0\in \sigma(A)$ which can always be achieved by adding a multiple of the identity. 
Let $(\varphi_m)_{m\in\N}$ be an orthonormal basis of $\H$ and $I_M:= \sum_{m=1}^M \bra{\varphi_m}\ket{\varphi_m}$ be the projection on the first $M$ basis vectors of this basis. For $M\in\N$, we define
\beq\label{def:restirc}
 \wtilde A_M:= I_M A I_M. 
 \eeq
Since we assumed $A$ and $B$ to be bounded, it follows directly that $\wtilde A_M$ converges strongly to $A$ as $M\to\infty$. Since $\partial I\subset \rho(A)\cap \rho(B)$ and $\partial I$ is a finite set, we obtain that 
\beq
\delta:=\dist(\partial I,\sigma(A)\cup \sigma(B))>0
\eeq
 and we choose
\beq
0<\varepsilon<\delta /2.
\eeq
We define an enlarged spectrum $\sigma_\varepsilon(A):=\{x\in\R:\, \dist(x,\sigma(A))\leq \varepsilon\}$ and set
\beq\label{def:AMBM}
A_M:=  \wtilde A_M\id_{\sigma_\varepsilon(A)}(\wtilde A_M)\qquad \text{and}\qquad B_M:= A_M +  I_M  \bra{\phi}\ket{\phi} I_M.
\eeq
Since $\partial \sigma_\varepsilon(A)\subset \rho(A)$, we have that $\id_{\sigma_\varepsilon(A)}(\wtilde A_M)\to \id$ strongly as $M\to\infty$, see e.g. \cite[Cor. 6.40]{MR2499016}.
Hence, $A_M$ converges strongly to $A$ as $M\to\infty$. Using the strong convergence $I_M\phi\to\phi$, we obtain as well that $B_M\to B$ strongly as $M\to\infty$. 

Let us first comment on the spectral structure of the operators $A_M$ and $B_M$. 
The construction of the operator $A_M$ implies that for $\eta:= \delta-\varepsilon>0$
\beq\label{pf:eq1}
\bigcup_{i=1}^{2N}(E_i-\eta, E_i+\eta) \subset \rho(A_M).
\eeq
  Here we need that $0\in \sigma(A)$. Otherwise, it might happen that by accident $0\in (E_i-\eta,E_i+\eta)$ for some $i\in \{1,...,2N\}$ and the inclusion \eqref{pf:eq1} might fail.
 In the following, we use the notation
\beq\label{def:xiM}
\xi_M(E):=\tr \big( \id_{(-\infty,E)}(A_M) - \id_{(-\infty,E)}(B_M) \big)\in  \{0,1\}
\eeq
for the spectral shift function corresponding to the pair $A_M, B_M$ of restricted operators at energy $E\in\R$. 
Concerning the spectrum of the operator $B_M$, we know the following:
 
\begin{lemma}\label{lemSSF}
Let $\eta:=\delta-\varepsilon$. Then there exists some $M_0\in\N$ such that for all $M\geq M_0$
\beq
\bigcup_{i=1}^{2N}[E_{i},E_{i} +\eta)\subset \rho(B_M).
\eeq
Moreover,  the spectral shift function satisfies $\xi_M(E)=0$ for all $E\in \bigcup_{i=1}^{2N}[E_{i},E_{i} +\eta)$ and all $M\geq M_0$. 
\end{lemma}

We accept this for the moment and prove it at the end of this paragraph. In particular, the above implies for all $M$ big enough that
\beq
\dist(I,(\sigma(A_M)\cup \sigma(B_M))\backslash I)\geq \eta.
\eeq
Hence, we are in position to apply the following general result. It allows for representing products of spectral projections as particular Bochner integrals:

\begin{lemma}\label{lem:intrep}
Let $C,D$ be two self-adjoint operators with $D-C:=V\in \S^1$. Let $J,K\in \Borel(\R)$ be two sets with $\dist(J,K)>0$. Let $f\in L^1(\R)$ such that $\hat f (s)= \frac 1 s$ for all $s\in J-K:=\{c-d:\ c\in J, d\in K\}$. Then 
\beq\label{bochner:eq1}
\id_J(C)\id_K(D) = \int_\R \d t\,  e^{-it C} \id_J(C)V \id_K(D) e^{itD} f(t)
\eeq
where the integral is a trace-class convergent Bochner integral. In particular, $\id_J(C)\id_K(D)$ is trace class. 
\end{lemma}

\begin{remark}
This identity is well-known and a main tool in estimating norms of spectral subspaces of pairs of self-adjoint operators, see e.g. \cite[Chap. 7]{MR1477662}.
\end{remark}

\begin{proof}
We set $Y:=\id_J(C) V\id_K(D)$. Since $\id_J(C)\id_K(D)$ solves the Sylvester equation $XD-CX=Y$ strongly, we obtain from e.g. \cite[Thm. 3.1]{MR3231246} that the identity \eqref{bochner:eq1} holds where the integral is a priori understood in the weak sense. 
However, the bound
\beq
\| e^{-it C} \id_J(C)V \id_K(D) e^{itD} \|_1=\|V\|_1 
\eeq
holds, where $\|\cdot\|_1$ stands for the trace norm. This together with $f\in L^1(\R)$ implies that the integral in \eqref{bochner:eq1} is not just a weakly convergent but a well-defined trace-class convergent Bochner integral and the identity \eqref{bochner:eq1} follows. 
\end{proof}

\begin{proposition}\label{prop:conv1}
Let $I$ be the set given in Theorem \ref{thm:main}. Then we obtain that the limits
\beq\label{lem1:eq1}
\lim_{M\to\infty} \id_{I}(A_M)\id_{I^c}(B_M) = \id_{I}(A)\id_{I^c}(B)
\eeq
and
\beq\label{lem1:eq3}
\lim_{M\to\infty} \id_{I^c}(A_M)\id_{I}(B_M) = \id_{I^c}(A)\id_{I}(B)
\eeq
exist in trace class.
In particular, $\id_{I}(A)\id_{I^c}(B),\, \id_{I^c}(A)\id_{I}(B)\in \S^1$. Moreover, 
\beq\label{lem1:eq2}
\lim_{M\to\infty} \det\big(\id - \id_{I}(A_M)\id_{I^c}(B_M)\id_{I}(A_M) \big) = \det\big(\id - \id_{I}(A)\id_{I^c}(B)\id_{I}(A) \big)
\eeq
and 
\beq\label{lem1:eq4}
\lim_{M\to\infty} \det\big(\id - \id_{I^c}(A_M)\id_{I}(B_M)\id_{I^c}(A_M) \big) = \det\big(\id - \id_{I^c}(A)\id_{I}(B)\id_{I^c}(A) \big).
\eeq
\end{proposition}

\begin{proof}
We set $Z:= \bigcup_{i=1}^N [E_{2i-1}+\eta, E_{2i} -\eta]$. 
Using \eqref{pf:eq1}, we obtain that
\beq
 \id_{I}(A_M)\id_{I^c}(B_M) = \id_Z(A_M) \id_{I^c}(B_M).
\eeq
We fix from now on $M\geq M_0$, where $M_0$ is given in Lemma \ref{lemSSF}. Since $\dist(Z,I^c)\geq \eta$, 
 Lemma \ref{lem:intrep} implies the integral representation
\beq\label{intrep2}
\id_Z(A_M) \id_{I^c}(B_M) = \int_\R \d t\,  e^{-it A_M} \id_Z(A_M)V \id_{I^c}(B_M) e^{itB_M} f(t),
\eeq
where $f\in L^1(\R)$ such that $\hat f (s)= \frac 1 s$ for all $s\in Z-I^c$. By definition, we have that $\partial Z\subset \rho(A)$ and $\partial I^c\subset \rho(B)$. Using \cite[Cor. 6.40]{MR2499016}, this implies that $\id_Z(A_M)\to \id_Z(A)$ and $\id_{I^c}(B_M)\to \id_{I^c}(B)$ strongly as $M\to\infty$. 
Since  $e^{it(\cdot)}$ is continuous, we obtain the strong convergence
\beq
\begin{aligned}
e^{-it A_M} \id_Z( A_M) \to e^{-itA} \id_Z(A) \quad\text{and}\quad
 e^{-it  B_M} \id_{I^c}( B_M)  \to e^{-it B}  \id_{I^c}( B)
\end{aligned}
\eeq
as $M\to\infty$, see e.g. \cite[Thm. 6.31]{MR2499016}. Applying the latter to $\phi$ gives that
\beq\label{lm:dom1}
 e^{-it A_M} \id_Z( A_M)|\phi\>\<\phi| \id_{I^c}(B_M) e^{it B_M} \to e^{-it A} \id_Z(A)|\phi\>\<\phi| \id_{I^c}(B) e^{itB} 
\eeq
in trace norm as $M\to\infty$. 
Additionally, we have the $M$-independent trace-class bound
\beq\label{lm:dom2}
\| e^{-itA_M} \id_I(A_M)V \id_{I^c}(B_M) e^{it B_M} f(t)\|_1 \leq \|\varphi\|^2 |f(t)|,
\eeq
where we recall that $f\in L^1(\R)$. Now, the convergence \eqref{lem1:eq1} follows from the integral representation \eqref{intrep2} and dominated convergence for Bochner integrals using \eqref{lm:dom1} and \eqref{lm:dom2}.

To prove \eqref{lem1:eq3} , we define $Z':= \R\backslash  \bigcup_{i=1}^N [E_{2i-1}-\eta, E_{2i} +\eta]$. Then,  \eqref{pf:eq1} implies
\beq
 \id_{I^c}(A_M)\id_{I}(B_M) = \id_{Z'}(A_M) \id_{I}(B_M),
\eeq
 where $\dist(I, Z')\geq \eta$. The rest of the proof follows along the same lines as above. 
 
The convergence  in \eqref{lem1:eq2} and \eqref{lem1:eq4} follow from continuity of Fredholm determinants with respect to trace norm, see \cite[Thm. 3.4]{MR2154153}.
\end{proof}

The last lemma establishes convergence of Fredholm determinants and we can approximate the l.h.s. of \eqref{thm1:eq2a} by finite-rank operators. To prove Theorem \ref{thm:main} we also need convergence of the r.h.s. of \eqref{thm1:eq2a} which we obtain from vague convergence of $\xi_M$ to $\xi$:

\begin{lemma}\label{lemma2}
Let $f\in C_c(\R)$, where $ C_c(\R)$ denotes the set of all compactly supported continuous functions. Then we have  the convergence
\beq
\lim_{M\to\infty}\int_\R \d x f(x) \xi_M(x) = \int_\R \d x f(x) \xi(x).
\eeq
\end{lemma}

\begin{remark}
Related convergence results for the spectral shift function can be found in \cite{MR2596053, MR2947287}. Since we are considering rank-one perturbations only, we give here a short and elementary proof for completeness.
\end{remark}
 
\begin{proof}
Define $B_M(s):= A_M + s  |\id_M\phi\>\<\id_M\phi|$,  $B(s):= A+ s|\phi\>\<\phi|$  and let $f\in C_c(\R)$. 
The Birman-Solomyak formula implies
\beq
\int_\R \d x\, f(x) \xi_M(x) = \int_0^1 \d s\,  \<\id_M\phi,  f(B_M(s))\id_M\phi\>,
\eeq
see e.g. \cite{MR1443857}. Since $f\in C_c(\R)$, the strong convergence $B_M(s) \to B(s)$ holds as $M\to\infty$. Moreover, $\id_M\to \id$ stongly as well as $M\to\infty$ and we obtain that the integrand in the latter converges, i.e.
\beq
\<\id_M\phi, f(B_M(s) )\id_M\phi\>  \to  \<\phi, f( B(s) )\phi\>
\eeq
as $M\to\infty$. Furthermore, the integrand is bounded by
\beq
\<\id_M\phi, f(B_M(s))\id_M\phi\> \leq \|f\|_\infty \|\phi\|^2
\eeq
independently of $s\in (0,1)$. 
Thus, the lemma follows from dominated convergence. 
\end{proof}

\begin{proposition}\label{prop:conv2}
Let the set $I$ be given as in Theorem \ref{thm:main}. Then we have the convergence
\beq
\lim_{M\to\infty} \int_I \d x\int_{I^c}\d y\, \frac{\xi_M(x)\xi_M(y)}{(x-y)^2}
=
\int_I \d x\int_{I^c}\d y\, \frac{\xi(x)\xi(y)}{(x-y)^2}.
\eeq
\end{proposition}

\begin{proof}
This follows from the vague convergence in Lemma \ref{lemma2} using that $\xi_M(x)= 0$ for all $x\in \bigcup_{i=1}^{2N}[E_i,E_i+\eta)$ and all $M\geq M_0$ which was proved in Lemma \ref{lemSSF}. 
\end{proof}

\begin{proof}[Proof of Lemma \ref{lemSSF}]
Fix  $i\in\{1,...,2N\}$. From our assumptions we know that $(E_i-\delta,E_i+\delta)\subset \rho(A)\cap\rho(B)$ and that $\xi(E_i)=0$. Since the spectral shift function is constant on any connected component of $\rho(A)\cap\rho(B)\cap \R$, we obtain that 
\beq\label{pf:ssf:eq1}
\xi(x)=0 \qquad\text{for all} \qquad x\in (E_i-\delta,E_i+\delta).
\eeq

Using that $A_M$ differs from $B_M$ by a rank-one perturbation and that $(E_i-\eta,E_i+\eta)\subset \rho(A_M)$ for all $M\in\N$, it follows that at most one eigenvalue of $B_M$ lies in the interval $(E_i-\eta,E_i+\eta)$. 
Assume by contradiction that there exists a sequence $( b_{M_n})_{n\in\N}$ of eigenvalues of $B_{M_n}$ such that $ b_{M_n}\in (E_i-\eta,E_i+\eta)$ and $ b_{M_n}\to  b\in  (E_i-\eta,E_i+\eta)$ as $n\to\infty$. Since $A_{M_n}$ and $B_{M_n}$ have discrete spectrum, it follows that $\xi_{M_n}(x)=1$ for all $x\in (E_i-\eta,  b_{M_n})$. The latter and Lemma \ref{lemma2} then yield for any $ f\in C_c((E_i-\eta,E_i+\eta))$ with $0\leq f \leq 1$ and $f\neq 0$ that
\begin{align}
\int_{E_i-\eta}^{E_i+\eta} \d x\, \xi(x) 
&\geq  \lim_{n\to\infty}\int_{E_i-\eta}^{E_i+\eta} \d x\, f(x) \xi_{M_n}(x)\notag \\
&\geq 
\liminf_{n\to\infty} \int_{E_i-\eta}^{ b_{M_n}} \d x\, f(x) \xi_{M_n}(x)
 = \int_{E_i-\eta}^b \d x\, f(x)>0. 
\end{align}
This contradicts \eqref{pf:ssf:eq1} and therefore,  for all $M$ big enough, $[E_i,E_i+\eta)\subset \rho(B_M)$ and $\xi_M=0$ for a set of positive Lebesgue measure in $[E_i,E_i+\eta)$. The spectral shift function is constant on any connected component of $\rho(A_M)\cap \rho(B_M)\cap \R$. Hence, it follows in particular that $\xi_M(x)=0$ for all $x\in [E_i,E_i+\eta)$. 
\end{proof}

\subsection{Proof of Theorem \ref{thm:main}, Corollary \ref{corollary}, Proposition \ref{prop1} and Theorem~\ref{prop2}}

\begin{proof}[Proof of Theorem \ref{thm:main}]
For the moment we additionally assume that $A,B$ are bounded. Let $A_M$ and $B_M$ be the finite-rank operators defined in \eqref{def:AMBM}. 
We fix $M\geq M_0$ where $M_0$ is given in Lemma \ref{lemSSF}. 
Then, the inclusion \eqref{pf:eq1} and Lemma \ref{lemSSF} imply that $\partial I\subset \rho(A_M)\cap \rho(B_M)$ and $\xi_M(x)=0$ for all $E\in\partial I$. 
Hence, the assumptions of Lemma \ref{main:lem} are satisfied and we obtain 
\beq\label{thm1:pr:eq2aa}
-\ln \left(\det \big(\id - \id_{I}(A_M) \id_{I^c}(B_M)\id_{I}(A_M)\big) \right)
=
\int_I \d x \int_{I^c} \d y\, \frac{\xi_M(x)\xi_M(y)}{(y-x)^2}.
\eeq

The theorem follows from taking the limit $M\to\infty$ in \eqref{thm1:pr:eq2aa} using Proposition~\ref{prop:conv1} and Proposition~\ref{prop:conv2}.

This proves the result for bounded operators. For unbounded operators which are bounded from below, let $c$ be constant such that $A> c$ and $B> c$.  Then, we apply the latter to their self-adjoint resolvents  $1/(A- c)$ and $1/(B-c)$ which are bounded operators. 
 \end{proof}

\begin{proof}[Proof of Theorem \ref{prop2}]
The proof of Theorem  \ref{prop2} follows along the same lines as  Theorem \ref{thm:main}:
For matrices it follows from  Lemma \ref{main:lem2}. The general Theorem \ref{prop2} follows from applying the result for matrices to the finite-rank approximation introduced in Section \ref{approx} and then taking the limit. In doing so one has to adapt Lemma \ref{lemSSF}. 
\end{proof}

\begin{proof}[Proof of Proposition \ref{prop1}]
Since $B-A$ is rank one and $\partial I\subset \rho(A)\cap\rho(B)$, it follows from Lemma \ref{lem:intrep}  that $\id_I(A) - \id_I(B)\in\S^1$ and its trace is given by
 \beq
 \tr (\id_I(A) - \id_I(B) ) = \ker (\id_I(A) - \id_I(B)-\id) - \ker (\id_I(A) - \id_I(B) + \id),
\eeq 
see \cite[Prop. 3.1]{MR1262254}. If $\tr (\id_I(A) - \id_I(B) ) >0$, this implies that $\ker (\id_I(A) - \id_I(B)-\id)\neq 0$. Let $\varphi\in \ker (\id_I(A) - \id_I(B)-\id)$. Then it is straight forward to check that $\varphi \in \ran \id_I(A)\cap \ker  \id_I(B)$ and hence $1\in \sigma(  \id_I(A)  \id_{I^c}(B)\id_I(A))$. This implies \eqref{prop1:eq1}. The identity \eqref{prop1:eq2} follows along the same lines. 
\end{proof}

\begin{proof}[Proof of Corollary \ref{corollary}]
This follows for matrices along the same lines as Lemma \ref{main:lem} using Lemma \ref{Fredholm:Repr}. By approximating general operators by matrices as done in Section \ref{approx},  the general result follows from the matrix case using Proposition \ref{prop:conv1} and Proposition~\ref{prop:conv2}.
\end{proof}

\subsection{Proof of Corollary \ref{subspace:thm} and Theorem \ref{subspace:thm2}}

\begin{proof}[Proof of Corollary \ref{subspace:thm}]
Since $\dist(\Sigma,\sigma(A)\backslash \Sigma)>0$ and $A$ is bounded, $\Sigma$ is of the form 
$\Sigma= \bigcup_{i=1}^{2N}[E_{2i-1},E_{2i}]$ for some $N\in \N$ and $E_1\leq E_2 < E_3\leq E_4< ... \leq E_{2N}$. 

Since $\Sigma_\delta\backslash \Sigma\subset \rho(A)$, we obtain that $\id_\Sigma(A) = \id_{\Sigma_\delta}(A)$ and therefore 
\beq
\id_\Sigma(A) - \id_{\Sigma_\delta}(B) = \id_{\Sigma_\delta}(A) - \id_{\Sigma_\delta}(B).
\eeq
The assumption $\|\phi\|^2<\delta$ yields
\beq\label{pf:subspace:eq30}
I_\delta:=\Sigma_\delta\backslash \Sigma_{\|\phi\|^2}\subset \rho(A)\cap \rho(B).
\eeq
Hence, Lemma \ref{lem:intrep} implies that $\id_{\Sigma_\delta}(A) \id_{\R\backslash \Sigma_\delta}(B)\in \S^1 $ and $\id_{\R\backslash \Sigma_\delta}(A) \id_{\Sigma_\delta}(B)\in S^1$  and  
\beq\label{pf:subspace:eq10}
(\id_{\Sigma_\delta}(A) - \id_{\Sigma_\delta}(B))^2\in \S^1
\eeq
by identity \eqref{proj-identity}.
Moreover, we obtain that
\beq\label{pf:subspace:eq3}
\xi(x)=0\qquad \text{ for all} \qquad x\in I_\delta.
\eeq
This follows for example from the Birman-Solomyak formula \cite{MR1443857}
\beq
\int_{I_\delta} \d x\, \xi(x) = \int_0^1 \d \lambda \,\<\phi, \id_{I_\delta} (A+ s|\phi\>\<\phi|) \phi\>. 
\eeq
Since $I_\delta \subset \rho(A+ s|\phi\>\<\phi|)$ for all $s\in [0,1]$ as well, we obtain that the latter integral is $0$ and $\xi(x) =0$ for Lebesgue almost all $x\in I_\delta$. Using that the spectral shift function is constant on connected components of $\rho(A)\cap\rho(B)\cap \R$, \eqref{pf:subspace:eq3} holds for all $E\in I_\delta$. 

It follows from \eqref{pf:subspace:eq10} that the equivalence
\beq\label{equivalence}
\|\id_{\Sigma_\delta}(A) - \id_{\Sigma_\delta}(B)\|<1 \quad \text{if and only if}\quad \det\big( \id - (\id_{\Sigma_\delta}(A) - \id_{\Sigma_\delta}(B))^2 \big) > 0
\eeq
holds. Now, \eqref{pf:subspace:eq30} and \eqref{pf:subspace:eq3} allow for applying Theorem \ref{thm:main} which gives
\begin{align}
\det(\id-(\id_{\Sigma_\delta}(A) - \id_{\Sigma_\delta}(B))^2)
&= \exp\Big(-2\int_{\Sigma_{\delta}} \dx \int_{\R\backslash \Sigma_{\delta}} \dy\, \frac{\xi(x)\xi(y)}{(x-y)^2}\Big)\notag\\
&= \exp\Big(-2\int_{\Sigma_{\|\phi\|^2}} \dx \int_{\R\backslash \Sigma_{\delta}} \dy\, \frac{\xi(x)\xi(y)}{(x-y)^2}\Big)
\end{align}
where the last equality follows from \eqref{pf:subspace:eq3}. Since $\dist(\Sigma_{\delta},\Sigma_{\|\phi\|^2})>0$ the latter integral is finite and the l.h.s. positive. Hence, \eqref{equivalence} gives the result. 
\end{proof}

\begin{proof}[Proof of Theorem \ref{subspace:thm2}]
The identity \eqref{proj-identity} implies that 
\beq
\|\id_I(A) - \id_I(B)\| \leq \max\{ \| \id_I(A) \id_{I^c}(B)\id_I(A)\|^{1/2},  \| \id_{I^c}(A)\id_I(B) \id_{I^c}(A) \|^{1/2} \}.
\eeq
Hence, it suficies to prove the result for the latter products. We perform the argument for $\| \id_I(A) \id_{I^c}(B)\id_I(A)\|$ only. The other term is estimated analogously. 
We define the spectral measures
\beq
\mu_A(\,\cdot\,) :=\<\phi, \id_{(\cdot)}(A) \phi\>\quad\text{and}\quad \mu_B(\,\cdot\,):=\<\phi, \id_{(\cdot)}(B) \phi\>
\eeq
and assume w.l.o.g. that $\phi$ is cyclic with respect to $A$. Hence, the spectral theorem in multiplication operator form gives that the mappings
\beq
\begin{aligned}
&U^*:\H\to L^2(\sigma(A), \d \mu_A),\quad h(A)\phi\mapsto h \\
&V^*:\H\to L^2(\sigma(B), \d \mu_B),\quad h(B)\phi\mapsto h 
\end{aligned}
\eeq
are well-defined unitary operators. 
Then we have  
\beq
U^*\id_{I}(A)\id_{I^c}(B)V = K,
\eeq
where $K:L^2(\sigma(B)\cap I^c, \d \mu_B)\to L^2(\sigma(A)\cap I, \d \mu_A)$ is the integral operator with kernel 
\beq
K(x,y) = \frac{1}{x-y}\qquad \text{for}\ x\in I,\ y\in I^c,
\eeq
see e.g. \cite[Thm. 2.1]{MR2540995}.
Using this, we obtain 
\beq
\|\id_{I}(A)\id_{I^c}(B)\id_{I}(A)\|   \leq \| K K^*\|.
\eeq
For $h\in L^2(\sigma(A)\cap I, \d \mu_A)$ we compute 
\begin{align}
\<K^* h, K^* h\> 
&= 
\Big|\int_{I^c\cap \sigma(B)} \d \mu_B(x) \int_{I\cap\sigma (A)} \d \mu_A(y) \frac { h(y)} {x-y} \int_{I\cap\sigma (A)} \d \mu_A(z) \frac { h(z)} {x-z} \Big|\notag\\
& \leq \frac 1 {\delta^2} \, \<\phi, \id_{I^c\cap \sigma(B)}(B) \phi\> \,  \<\phi, \id_{I\cap\sigma (A)}(A) \phi\> \,  \| h\|^2_{L^2(I^c \cap \sigma(B), \d \mu_B)}\notag\\
&\leq  \frac 1 {\delta^2} \, \|\phi\|^4 \| h\|^2_{L^2(I^c\cap \sigma(B), \d \mu_B)},
\end{align}
where we used that $\dist(I\cap \sigma(A),I^c\cap\sigma (B))\geq \delta$ and the Cauchy-Schwarz inequality. This gives the assertion. 
\end{proof}

\begin{appendix}
\section{Rank-one analysis and residues of resolvents}

In the appendix $A,B\in \C^{M\times M}$ are self-adjoint and $B:=A+ |\phi\>\<\phi|$ where $\phi\in \C^M$ and $M\in\N$. Moreover, we assume $\phi$ to be cyclic with respect to $A$. We use the notation of Section \ref{subsec:matrices}.

\begin{lemma}\label{cor:det}
 Let $j,k\in\N$. We obtain the identities
  \begin{align}
 \left|\<\varphi_j,\phi\>\right|^2 = \left( b_j- a_j\right)\prod_{\substack{l=1\\l\neq j}}^M
				    \frac{ \left( b_l- a_j\right)}{\left( a_l- a_j\right)}
\end{align}				
and
\begin{align}    
	 \left|\<\psi_k,\phi\>\right|^2&= \left( a_k- b_k\right)
 \prod_{\substack{l=1\\l\neq k}}^M\frac{ \left( a_l- b_k\right)}{ \left( b_l- b_k\right)}.
  \end{align}
\end{lemma}

\begin{proof}
Let $z\in\C\backslash \sigma(A)$. The identities
\begin{align}
 \left|\<\varphi_j,\phi\>\right|^2
 =\lim_{z\to a_j} \left( a_j-z\right)\<\phi,\frac 1 {A-z} \phi\>
 =\lim_{z\to a_j} \left( a_j-z\right)\big(1 + \<\phi,\frac 1 {A-z} \phi\>\big)\label{appendix:eq1}
\end{align}
hold. On the other hand, since $(B-z) (A-z)^{-1} = \id + V (A-z)^{-1} $ and $V=|\phi\>\<\phi|$ is a rank-one perturbation, we obtain 
\begin{align}
1 + \<\phi,\frac 1 {A-z} \phi\> = \det \big(\id + V (A-z)^{-1}\big)&= \det \big(B-z\big)\det\big( (A-z)^{-1}\big) \notag \\
&= 
\prod_{n=1}^N (b_n -z)(a_n-z)^{-1}.\label{appendix:eq2}
\end{align}
\eqref{appendix:eq1} and \eqref{appendix:eq2} together imply the assertion. 
\end{proof}

\end{appendix}

\section*{Acknowledgements}
M.G. thanks Alexander Pushnitski, Peter M\"uller and especially Adrian Dietlein for many illuminating discussions on this topic.

\newcommand{\etalchar}[1]{$^{#1}$}

\end{document}